\documentclass[a4paper]{amsart}

    \usepackage{etex}
    
\usepackage{hyperref}
\usepackage{txfonts, amsmath,amstext,amsthm,amscd,amsopn,verbatim,amssymb, amsfonts, amsrefs}
\usepackage{fullpage}

\usepackage[bbgreekl]{mathbbol}
\usepackage{tikz}
\usetikzlibrary{matrix}
\usetikzlibrary{shapes}
\usetikzlibrary{arrows}
\usetikzlibrary{calc,3d}
\usetikzlibrary{decorations,decorations.pathmorphing}
\usetikzlibrary{through}
\tikzset{ext/.style={circle, draw,inner sep=1pt},int/.style={circle,draw,fill,inner sep=1pt},nil/.style={inner sep=1pt}}
\tikzset{exte/.style={circle, draw,inner sep=3pt},inte/.style={circle,draw,fill,inner sep=3pt}}
\tikzset{diagram/.style={matrix of math nodes, row sep=3em, column sep=2.5em, text height=1.5ex, text depth=0.25ex}}
\tikzset{diagram2/.style={matrix of math nodes, row sep=0.5em, column sep=0.5em, text height=1.5ex, text depth=0.25ex}}

\usepackage{tikz-cd}

\theoremstyle{plain}
  \newtheorem{thm}{Theorem}[section]
   
  \newtheorem{defi}[thm]{Definition}
  \newtheorem{prop}[thm]{Proposition}
  
  \newtheorem{cor}[thm]{Corollary}
  \newtheorem{lemma}[thm]{Lemma}
   
\theoremstyle{definition}
  \newtheorem{ex}[thm]{Example}
  \newtheorem{rem}[thm]{Remark}

\newcommand{\alg}[1]{\mathfrak{{#1}}}

\newcommand{\ad}{{\text{ad}}}

\newcommand{\R}{{\mathbb{R}}}

\newcommand{\K}{{\mathbb{K}}}
\newcommand{\Q}{{\mathbb{Q}}}

\newcommand{\HGC}{{\mathrm{HGC}}}

\newcommand{\Graphs}{{\mathsf{Graphs}}}

\newcommand{\Conv}{\mathrm{Conv}}

\newcommand{\fHGC}{{\mathrm{fHGC}}}
\newcommand{\mF}{\mathcal{F}}

\newcommand{\bbS}{\mathbb{S}}

\newcommand{\Tw}{\mathit{Tw}}

\newcommand{\Poiss}{\mathsf{Pois}}

\newcommand{\op}{\mathcal}

\newcommand{\Lie}{\mathsf{Lie}}

\newcommand{\hoLie}{\mathsf{hoLie}}

\newcommand{\Ass}{\mathsf{Assoc}}
\newcommand{\Com}{\mathsf{Com}}

\newcommand{\FM}{\mathsf{FM}}

\newcommand{\conn}{\mathit{conn}}

\newcommand{\bpm}{\begin{pmatrix}}
\newcommand{\epm}{\end{pmatrix}}

\newcommand{\GC}{\mathrm{GC}}

\newcommand{\MC}{\mathsf{MC}}

\newcommand{\lD}{\mathsf{D}}

\newcommand{\Map}{\mathrm{Map}}

\newcommand{\beq}[1]{\begin{equation}\label{#1}}
\newcommand{\eeq}{\end{equation}}

\newcommand{\fg}{\mathfrak{g}}
\newcommand{\RT}{\mathsf{RT}}
\newcommand{\TRT}{\mathsf{TRT}}
\newcommand{\PL}{\mathsf{PL}}
\newcommand{\PLM}{\mathsf{PLM}}
\newcommand{\RTM}{\mathsf{RTM}}

\newcommand{\Exp}{\mathsf{Exp}}

\newcommand{\oMC}{\overline{\MC}}

\newcommand{\closed}{\mathrm{closed}}

\newcommand{\gexp}{\mathrm{gexp}}

\newcommand{\BCH}{\mathrm{BCH}}

\begin{document}
\title{Pre-Lie pairs and triviality of the Lie bracket on the twisted hairy graph complexes}

%

 \author{Thomas Willwacher}
\address{Department of Mathematics \\ ETH Zurich \\  
R\"amistrasse 101 \\
8092 Zurich, Switzerland}
\email{thomas.willwacher@math.ethz.ch}

\thanks{The author has been partially supported by the Swiss National Science foundation, grant 200021-150012, the NCCR SwissMAP funded by the Swiss National Science foundation, and the ERC starting grant GRAPHCPX (678156).}


\begin{abstract}
We study pre-Lie pairs, by which we mean a pair of a homotopy Lie algebra and a pre-Lie algebra with a compatible pre-Lie action. Such pairs provide a wealth of algebraic structure, which in particular can be used to analyze the homotopy Lie part of the pair. 

Our main application and the main motivation for this development are the dg Lie algebras of hairy graphs computing the rational homotopy groups of the mapping spaces of the $E_n$ operads. We show that twisting with certain Maurer-Cartan elements trivializes their Lie algebra structure. The result can be used to understand the rational homotopy type of many connected components of the mapping spaces of $E_n$ operads. 
\end{abstract}

\maketitle

\section{Introduction}
We consider the hairy graph complexes $\HGC_{m,n}$, whose elements are formal series of isomorphism classes of undirected graphs with ``hairs'' (or ``legs'', as physicists would prefer), for example
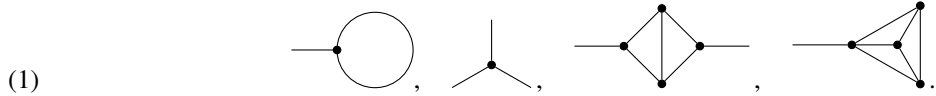
\beq{equ:hairysample}
 \begin{tikzpicture}[scale=.5]
 \draw (0,0) circle (1);
 \draw (-180:1) node[int]{} -- +(-1.2,0);
 \end{tikzpicture}
,\quad
\begin{tikzpicture}[scale=.6]
\node[int] (v) at (0,0){};
\draw (v) -- +(90:1) (v) -- ++(210:1) (v) -- ++(-30:1);
\end{tikzpicture}
\,,\quad
\begin{tikzpicture}[scale=.5]
\node[int] (v1) at (-1,0){};\node[int] (v2) at (0,1){};\node[int] (v3) at (1,0){};\node[int] (v4) at (0,-1){};
\draw (v1)  edge (v2) edge (v4) -- +(-1.3,0) (v2) edge (v4) (v3) edge (v2) edge (v4) -- +(1.3,0);
\end{tikzpicture}
 \, ,\quad
 \begin{tikzpicture}[scale=.6]
\node[int] (v1) at (0,0){};\node[int] (v2) at (180:1){};\node[int] (v3) at (60:1){};\node[int] (v4) at (-60:1){};
\draw (v1) edge (v2) edge (v3) edge (v4) (v2)edge (v3) edge (v4)  -- +(180:1.3) (v3)edge (v4);
\end{tikzpicture}
 \, .
 \eeq
The differential on $\HGC_{m,n}$ is given by vertex splitting. Furthermore, there is a compatible Lie bracket given combinatorially by attaching a hair of one graph to another:
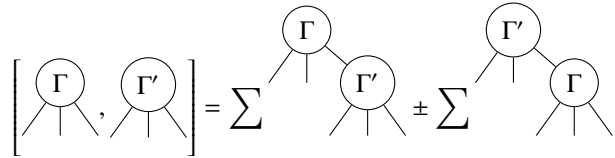
\[
\left[ 
\begin{tikzpicture}[baseline=-.8ex]
\node[draw,circle] (v) at (0,.3) {$\Gamma$};
\draw (v) edge +(-.5,-.7) edge +(0,-.7) edge +(.5,-.7);
\end{tikzpicture}
,
\begin{tikzpicture}[baseline=-.65ex]
\node[draw,circle] (v) at (0,.3) {$\Gamma'$};
\draw (v) edge +(-.5,-.7) edge +(0,-.7) edge +(.5,-.7);
\end{tikzpicture}
\right]
=
\sum
\begin{tikzpicture}[baseline=-.8ex]
\node[draw,circle] (v) at (0,1) {$\Gamma$};
\node[draw,circle] (w) at (.8,.3) {$\Gamma'$};
\draw (v) edge +(-.5,-.7) edge +(0,-.7) edge (w);
\draw (w) edge +(-.5,-.7) edge +(0,-.7) edge +(.5,-.7);
\end{tikzpicture}
\pm 
\sum
\begin{tikzpicture}[baseline=-.8ex]
\node[draw,circle] (v) at (0,1) {$\Gamma'$};
\node[draw,circle] (w) at (.8,.3) {$\Gamma$};
\draw (v) edge +(-.5,-.7) edge +(0,-.7) edge (w);
\draw (w) edge +(-.5,-.7) edge +(0,-.7) edge +(.5,-.7);
\end{tikzpicture}
\]

If $m=n$ there is the following Maurer-Cartan element with one edge and no vertices.
\beq{equ:linegraph}
L= 
\begin{tikzpicture}[baseline=-.65ex]
\draw (0,-.2) -- (1,0.2);
\end{tikzpicture}
 \in \HGC_{n,n}
\eeq
Consider the twisted dg Lie algebra $\HGC_{n,n}^L$. The following result is known.
\begin{thm}[\cite{Will, FW}]\label{thm:pre1}
There is a quasi-isomorphism of complexes 
\beq{equ:Lmap}
\K[1]\oplus \GC_n^2[1] \to \HGC_{n,n}^L,
\eeq
where $\GC_n^2$ is the non-hairy graph complex, cf. \cite{TW} or see section \ref{sec:HGC} below.
\end{thm}

Similarly, for $m=n-1$ there is also a Maurer-Cartan element
\beq{equ:tripodMC}
T:=  \sum_{k\geq 1} 
 \underbrace{
\begin{tikzpicture}[baseline=-.65ex]
\node[int] (v) at (0,0) {};
\draw (v) edge +(-.5,-.5)  edge +(-.3,-.5) edge +(0,-.5) edge +(.3, -.5) edge +(.5,-.5);
\end{tikzpicture}
}_{2k+1}
\in \HGC_{n-1,n}. 
\eeq
The following result has been shown in \cite{TW}.

\begin{thm}[\cite{TW}]\label{thm:pre2}
There is a quasi-isomorphism of complexes 
\beq{equ:Tmap}
\K[1]\oplus \GC_n^2[1] \to \HGC_{n-1,n}^T.
\eeq
\end{thm}

Finally, suppose that $m=1$. Then there is a special $L_\infty$-structure on $\HGC_{1,n}$, called the Shoikhet $L_\infty$-structure, such that the nerve of $\HGC_{1,n}$ with that $L_\infty$-structure is weakly equivalent to the (derived) mapping space of the little discs operads $\Map^h(\lD_1\to  \lD_n^{\Q})$ \cite{FTW}, where $\lD_n$ shall denote the little $n$-discs operad and $\lD_n^\Q$ its rationalization. Let us denote the space $\HGC_{1,n}$ equipped with the Shoikhet $L_\infty$-structure by $\HGC_{1,n}'$.

There is hence a Maurer-Cartan element $T'\in \HGC_{1,2}'$ which corresponds to the "natural" embedding $\lD_1\to \lD_2$. Concretely, as seen in \cite{TW}
\[
T'= \frac 1 4
\begin{tikzpicture}[baseline=-.65ex]
\node[int] (v) at (0,0) {};
\draw (v) edge +(-.5,-.5)  
edge +(0,-.5) 
edge +(.5,-.5);
\end{tikzpicture}
+(\cdots)
\]
The following result can be obtained from \cite{TW}.
\begin{thm}[\cite{TW}]\label{thm:pre3}
There is a quasi-isomorphism of complexes 
\beq{equ:TTmap}
\K[1]\oplus \GC_2^2[1] \to (\HGC_{1,2}')^{T'}.
\eeq
\end{thm}

The main result of this short note is the following.
\begin{thm}\label{thm:main}
The maps \eqref{equ:Lmap}, \eqref{equ:Tmap} and \eqref{equ:TTmap} can be extended to $L_\infty$-quasi-isomorphisms
\begin{align}
U \colon & \K[1]\oplus \GC_n^2[1] \to \HGC_{n,n}^L \\
V \colon & \K[1]\oplus \GC_n^2[1] \to \HGC_{n-1,n}^T \\
V' \colon & \K[1]\oplus \GC_2^2[1] \to (\HGC_{1,2}')^{T'},
\end{align}
where the left hand side is considered as an abelian dg Lie algebra, i.e., equipped with the trivial Lie bracket.
In particular, in each case the right-hand side is formal as an $L_\infty$-algebra.
\end{thm}

This result is used in \cite{FTW} to compute the homotopy type of (certain components of) the function spaces $\Map^h(\lD_m,\lD_n^{\Q})$ in codimension $n-m=0$ and $n-m=1$.

We will derive Theorem \ref{thm:main} in a very general algebraic context, that is of some interest in its own right, and that we expect to be applicable to other problems in deformation theory as well.
To this end we introduce the notion of pre-Lie pair $(\alg g,M)$, by which we mean a pre-Lie algebra $\alg g$ and an $L_\infty$-algebra $M$, with a compatible right pre-Lie action of $\alg g$ on $M$. This specializes the more standard example of $\alg g$ being a Lie algebra. However, the pre-Lie structure allows for certain new algebraic constructions.
In particular, if $\alpha\in \alg g$ is a Maurer-Cartan element and $m\in M^\alpha$ is a Maurer-Cartan element in the $\alpha$-twist $M^\alpha$ of $M$, then we can define a natural $L_\infty$-structure on $\alg g^\alpha[1]$ together with a natural $L_\infty$-morphism 
\[
 \alg g^\alpha[1] \to M^{\alpha,m}
\]
into the $L_\infty$-algebra $M^{\alpha,m}$ obtained by twisting $M^\alpha$ by $m$.
Furthermore, using the pre-Lie structure on $\alg g$ the $L_\infty$-structure on $\alg g^\alpha[1]$ may be trivialized, thus uncovering a trivial $L_\infty$-subalgebra of $M^{\alpha,m}$.

In the examples of the hairy graph complexes it turns out that this trivial $L_\infty$-subalgebra is in fact quasi-isomorphic to the whole $L_\infty$-algebra thus giving rise to our Triviality Theorem \ref{thm:main}.

\section{Basic notation}
We generally work over a ground field $\K$ of characteristic 0. All vector spaces and differential graded (dg) vector spaces are to be understood as $\K$-vector spaces.
We work in homological conventions, so that the differentials have degree $-1$. Generally, the differential on dg vector spaces will be denoted by $d$. For $V$ a differential graded vector space we define $V[k]$ such that an element $v\in V$ of degree $j$ has degree $j-k$ as an element of $V[k]$. 

We will often use the language of operads and colored operads.
A good introduction may be found in the textbook \cite{LV}. In particular, we will denote by $\Lie$ the operad governing Lie algebras. We denote by $\op P\{k\}$ the operad whose algebras are vector spaces $V[k]$, where $V$ is a $\op P$-algebra. 
The Koszul dual of a quadratic operad is denoted by $\op P^\vee$. The cobar construction of a coaugmented cooperad $\op C$ is denoted by $\Omega(\op C)$. The operad governing $L_\infty$-algebras is denoted by 
\[
\hoLie = \Omega(\Lie^\vee).
\]
We will be slightly inconsistent in the notation in that we use both the (equivalent) words "$L_\infty$-algebra" and "$\hoLie$-algebra" for the same object.
We also define the degree shifted version  
\[
\hoLie_2 := \hoLie\{1\}.
\]

\section{Pre-Lie algebras and their modules}
Let us recall the notion of pre-Lie algebras and modules. For a more detailed account and history we refer to \cite{CL} and references therein.
\begin{defi}\label{def:prelie}
A (dg) pre-Lie algebra is a (dg) vector space $\alg g$ together with a binary operation 
\[
\bullet : \alg g\otimes \alg g \to \alg g
\]
such that for all $x,y,z\in \fg$ we have
\[
x\bullet (y\bullet z) - (x\bullet y) \bullet z = x\bullet (z\bullet y) - (x\bullet z) \bullet y.
\]
A right module for the pre-Lie algebra $\alg g$ is a vector space $M$ together with an operation 
\[
\circ : M\otimes \fg \to M
\]
such that for all $m\in M$, $x,y\in \fg$ we have
\[
m\circ (x\bullet y) - (m\circ x) \bullet y = m\circ (y\bullet x) - (m\circ y) \bullet x.
\]
We equip the tensor product $M\otimes M'$ of two right $\fg$-modules $M,M'$ with a right $\fg$-module structure such that for homogeneous $m\in M$, $m'\in M'$, $x\in\fg$
\[
(m\otimes m')\circ x = (-1)^{|x||m'|}(m\circ x)\otimes m' + m\otimes (m'\circ x).
\]
Suppose now that $M$ is an addition equipped with an algebra structure over the operad $\op P$. Then we say that the operadic action and $\fg$-action are compatible if they equip $M$ with the structure of $\op P$ algebra in the symmetric monoidal category of right $\fg$-modules. Concretely, this means that for all homogeneous $p\in \op P(r)$, $m_1,\dots,m_r\in M$ and $x\in \fg$
\[
p(m_1,\dots ,m_r)\circ x = p(m_1,\dots, m_r \circ x) \pm p(m_1,\dots ,m_{r-1} \circ x,m_r) +\dots\pm p(m_1 \circ x,\dots, m_r).
\]
In the special case of $\op P$ being the operad governing $L_\infty$-algebras we call the pair $(\alg g,M)$ a \emph{pre-Lie pair}.
\end{defi}

A pre-Lie algebra $\fg$ is in particular a dg Lie algebra with the bracket
\[
[x,y] := x\bullet y - (-1)^{|x||y|}y\bullet x.
\]
A pre-Lie module is then automatically a dg Lie algebra module.

Note that the operation of taking the Lie bracket $[x,-]$ with some fixed $x$ in a pre-Lie algebra $\alg g$ is generally not a derivation with respect to the pre-Lie product.
Instead, one has the following simple variation.
\begin{lemma}\label{lem:preliederivation}
 Let $\alg g$ be a pre-Lie algebra and let $x,y_1,\dots, y_n\in\alg g$ be any homogeneous elements. Then 
 \begin{multline*}
  (\cdots(x\bullet y_1) \bullet y_1) \bullet \cdots \bullet y_n)
  -
  (-1)^{|x|(|y_1|+\dots+|y_n|)}
  (\cdots(y_1\bullet y_2) \bullet y_3) \bullet \cdots \bullet y_n) \bullet x
  \\=
  \sum_{j=1}^n
  (-1)^{|x|(|y_1|+\dots+|y_{j-1}|}
  (\cdots(y_1\bullet y_2) \bullet y_3) \bullet \cdots \bullet [x,y_j])\bullet \cdots  \bullet y_{n-1})\bullet y_n.
 \end{multline*}
\end{lemma}
\begin{proof}
 In general 
 \[
  z \bullet [x,y] = ((z\bullet x)\bullet y) - (-1)^{|x||y|} ((z\bullet y)\bullet x).
 \]
Inserting this on the right-hand side of the equation in the lemma one sees that the sum collapses telescopically to the left-hand side.
\end{proof}

\subsection{Operad of pre-Lie algebras and rooted trees}
We recall from \cite{CL} the definition of the rooted tree operad $\RT$.
The space $\RT(n)$ consists of linear combinations of rooted trees with vertices labelled by numbers $1,2,\dots,n$, e.g.,
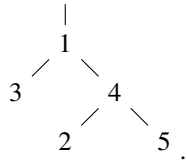
\[
\begin{tikzpicture}[scale=1.3]
\node (r) at (0,0) {$1$};
\node (v1) at (-.5,-.5) {$3$};
\node (v2) at (.5,-.5) {$4$};
\node (v3) at (0,-1) {$2$};
\node (v4) at (1,-1) {$5$};
\draw (r) edge +(0,.4) edge (v2) edge (v1)
          (v2) edge (v3) edge (v4);
\end{tikzpicture}
.
\]
The composition is defined by inserting one tree into one vertex of another, and summing over all ways of reconnecting the edges. We describe this pictorially
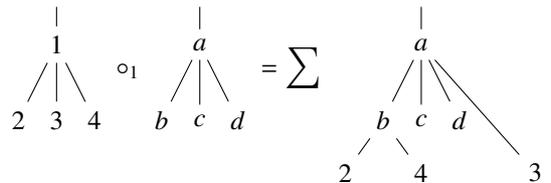
\[
\begin{tikzpicture}[baseline=-.65ex]
\node (r) at (0,.3) {$1$};
\node (v1) at (-.5,-.7) {$2$};
\node (v2) at (.5,-.7) {$4$};
\node (v3) at (0,-.7) {$3$};
\draw (r) edge +(0,.5) edge (v1) edge (v2) edge (v3);
\end{tikzpicture}
\circ_1 
\begin{tikzpicture}[baseline=-.65ex]
\node (r) at (0,.3) {$a$};
\node (v1) at (-.5,-.7) {$b$};
\node (v2) at (.5,-.7) {$d$};
\node (v3) at (0,-.7) {$c$};
\draw (r) edge +(0,.5) edge (v1) edge (v2) edge (v3);
\end{tikzpicture}
=\sum \,
\begin{tikzpicture}[baseline=-.65ex]
\node (r) at (0,.3) {$a$};
\node (v1) at (-.5,-.7) {$b$};
\node (v2) at (.5,-.7) {$d$};
\node (v3) at (0,-.7) {$c$};
\node (w1) at (0,-1.4) {$4$};
\node (w2) at (-1,-1.4) {$2$};
\node (w3) at (1.5,-1.4) {$3$};
\draw (r) edge +(0,.5) edge (v1) edge (v2) edge (v3) edge (w3)
         (v1) edge (w2) edge (w1);
\end{tikzpicture}
,
\]
and otherwise refer to \cite{CL} for more details.
The main result of loc. cit. is the following.
\begin{thm}[\cite{CL}]
The operad of rooted trees $\RT$ is identical to the operad $\PL$ governing pre-Lie algebras. The isomorphism
\[
\PL\to \RT
\]
is realized by sending the generator $(-\bullet -)$ to the tree
\[
\begin{tikzpicture}[baseline=-.65ex]
\node (v) at (0,.3) {$1$};
\node (w) at (0,-.3) {$2$};
\draw (v) edge (w) edge +(0,.5);
\end{tikzpicture} .
\]
\end{thm}

Similarly, let $\PLM$ be the two-colored operad governing a pair consisting of a pre-Lie algebra and a right pre-Lie module. Let similarly $\RTM$ be the two colored operad such that the color 1 piece $\RTM^1(r,0)$ is $\RT$, and such that 
$\RTM^2(r,1)$ is the space of linear combinations of rooted trees with top vertex marked $*$ and the other vertices numbered $1,\dots,r$, e.g., 
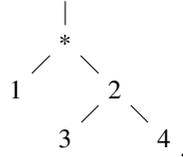
\[
\begin{tikzpicture}[scale=1.3]
\node (r) at (0,0) {$*$};
\node (v1) at (-.5,-.5) {$1$};
\node (v2) at (.5,-.5) {$2$};
\node (v3) at (0,-1) {$3$};
\node (v4) at (1,-1) {$4$};
\draw (r) edge +(0,.4) edge (v2) edge (v1)
          (v2) edge (v3) edge (v4);
\end{tikzpicture}
.
\]
By a mild extension of the result of \cite{CL} we obtain the following.
\begin{thm}[extension of \cite{CL}]
The two-colored operads $\RTM$ and $\PLM$ are identical. The isomorphism
\[
\PLM\to \RTM
\]
is realized by sending the generators $(-\bullet -)$ and $(-\circ -)$ to the following trees
\begin{align*}
(-\bullet -) & \mapsto
\begin{tikzpicture}[baseline=-.65ex]
\node (v) at (0,.3) {$1$};
\node (w) at (0,-.3) {$2$};
\draw (v) edge (w) edge +(0,.5);
\end{tikzpicture} 
& \text{and} & & 
(-circ -) & \mapsto
\begin{tikzpicture}[baseline=-.65ex]
\node (v) at (0,.3) {$*$};
\node (w) at (0,-.3) {$1$};
\draw (v) edge (w) edge +(0,.5);
\end{tikzpicture} .
\end{align*}
\end{thm}

We introduce the following notation for certain special operations in $\RTM$. First, the operation 
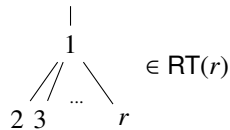
\[
\begin{tikzpicture}[baseline=-.65ex]
\node (r) at (0,.3) {$1$};
\node (v1) at (-.7,-.7) {$2$};
\node (v2) at (-.4,-.7) {$3$};
\node at (.1,-.5) {$\scriptstyle \cdots$};
\node (v3) at (0.7,-.7) {$r$};
\draw (r) edge +(0,.5) edge (v1) edge (v2) edge (v3);
\end{tikzpicture}
 \in \RT(r)
\]
will be denoted on elements $x_1,\dots, x_r\in \fg$ by 
\beq{equ:multibullet}
x_1 \bullet (x_2,\dots,x_r).
\eeq
Second, the operation 
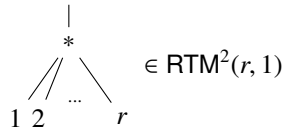
\[
\begin{tikzpicture}[baseline=-.65ex]
\node (r) at (0,.3) {$*$};
\node (v1) at (-.7,-.7) {$1$};
\node (v2) at (-.4,-.7) {$2$};
\node at (.1,-.5) {$\scriptstyle \cdots$};
\node (v3) at (0.7,-.7) {$r$};
\draw (r) edge +(0,.5) edge (v1) edge (v2) edge (v3);
\end{tikzpicture} \in \RTM^2(r,1)
\]
will be denoted on elements $m\in M$, $x_1,\dots, x_r\in \fg$ by
\beq{equ:circnotation}
m\circ (x_1,\dots,x_r)
\eeq

Finally let us note that the fact that a pre-Lie algebra is a Lie algebra can be expressed by the existence of an operad map $\Lie\to \RT$, defined on the generator as follows:
\beq{equ:mapfromlie}
[-, -] \mapsto 
\begin{tikzpicture}[baseline=-.65ex]
\node (v) at (0,.3) {$1$};
\node (w) at (0,-.3) {$2$};
\draw (v) edge (w) edge +(0,.5);
\end{tikzpicture} 
-
\begin{tikzpicture}[baseline=-.65ex]
\node (v) at (0,.3) {$2$};
\node (w) at (0,-.3) {$1$};
\draw (v) edge (w) edge +(0,.5);
\end{tikzpicture} 
\eeq

Finally, for later use, let us show the following elementary formulas.
\begin{lemma}\label{lem:distributive}
 Let $\fg$ be a pre-Lie algebra acting on $M$. Let $x\in \fg$ be of degree 0 and let $\lambda$ be a formal parameter. Define the operation 
 \begin{gather*}
  E_{\lambda x} \colon M[[\lambda]] \to M[[\lambda]] \\
  m \mapsto \sum_{j\geq 1} \frac{\lambda^j}{j!} 
  \underbrace{(\cdots(m\circ x)\circ x)\circ \cdots \circ x)\circ x}_{j\text{ many ``$x$''}}.
 \end{gather*}
Furthermore define 
\[
 e_{\lambda x} =  \sum_{j\geq 1} \frac{\lambda^j}{j!} 
  \underbrace{(\cdots(x\bullet x)\bullet x)\bullet \cdots \bullet x)\bullet x}_{j\text{ many ``$x$''}} \in\fg[[\lambda]].
\]
Then the following holds for all $m\in M, x_1,\dots,x_r\in\fg$:
\begin{align}\label{equ:temp8}
 E_{\lambda x} m &= \sum_{j\geq 1} \frac 1 {j!} m\circ \underbrace{(e_{\lambda x},\dots,e_{\lambda x} )}_{j\times},
\\
\label{equ:temp9}
 E_{\lambda x} (m\circ(x_1,\dots,x_r))+m\circ(x_1,\dots,x_r) &= \sum_{j\geq 0} \frac 1 {j!} m\circ (E_{\lambda x}x_1+x_1,\dots,E_{\lambda x}x_r+x_r,\underbrace{e_{\lambda x},\dots,e_{\lambda x} )}_{j\times},
\end{align}
where we used the notation \eqref{equ:circnotation}, and considered $\fg$ as a pre-Lie module over $\fg$ for defining $E_{\lambda x}x_j$.
\end{lemma}
\begin{proof}
 Clearly both sides of each of the equations agree in $\lambda=0$. Furthermore, we show that both sides satisfy the same (formal) ordinary differential equation in $\lambda$. More concretely, since $E_{\lambda x}$ is the exponential of the right action with $\lambda x$, minus the identity:
 \[
  \frac{d}{d\lambda} (E_{\lambda x} m) = (E_{\lambda x} m) \circ x + m\circ x.
 \]
 On the other hand, 
 \[
\frac{d}{d\lambda}e_{\lambda x} = e_{\lambda x}  \bullet x + x.
 \]
Hence, abbreviating the right-hand side of \eqref{equ:temp8} by $A$ we find that
\begin{align*}
\frac{d}{d\lambda} A
 &=
\sum_{j\geq 1} \frac j {j!} m\circ (e_{\lambda x}\bullet x +x,\dots,e_{\lambda x} )
 \\
 &=
\sum_{j\geq 1} \frac j {j!} m\circ (e_{\lambda x}\bullet x,\dots,e_{\lambda x} )
+
\sum_{j\geq 1} \frac 1 {(j-1)!} m\circ (x,\underbrace{e_{\lambda x},\dots,e_{\lambda x}}_{j-t\times} )
\\
 &=
\sum_{j\geq 1} \frac 1 {j!}\left(
(m\circ (e_{\lambda x},\dots,e_{\lambda x} ))\circ x 
- m\circ (x,\underbrace{e_{\lambda x}\dots,e_{\lambda x}}_{j\times} )\right)
+
\sum_{j\geq 1} \frac 1 {(j-1)!} m\circ (x,\dots,e_{\lambda x} )
\\
 &=
 A\circ x
 +
 m\circ x.
\end{align*}
Hence the left-hand side and the right-hand side of \eqref{equ:temp8} satisfy the same ODE in $\lambda$, and by uniqueness of solutions we find that both sides agree.\footnote{To this end, the careful reader may want to replace $\lambda$ by $t\lambda$ for $t$ a non-formal variable and then apply the standard (non-formal) uniqueness result to the ODE in $t$.}

Turning to \eqref{equ:temp9} denote the left-hand side by $B$ and the right-hand side by $C$.
As before we easily find that 
\[
 \frac{d}{d\lambda} B = B\circ x.
\]
Now again by a similar calculation as before:
\begin{align*}
\frac{d}{d\lambda} C
 &=
 \sum_{j\geq 0}\sum_{i=1}^r \frac 1 {j!} m\circ (E_{\lambda x}x_1+x_1,\dots,(E_{\lambda x}x_i) \bullet x+x_i\bullet x, \dots,E_{\lambda x}x_n +x_n,e_{\lambda x},\dots,e_{\lambda x} )
 \\&\quad\quad\quad\quad +
\sum_{j\geq 0} \frac j {j!} m\circ (E_{\lambda x}x_1,\dots ,E_{\lambda x}x_n , e_{\lambda x}\bullet x +x,e_{\lambda x},\dots,e_{\lambda x} )
 \\
 &=
 C\circ x
 -
 \sum_{j\geq 0}\frac 1 {j!}
  m\circ (E_{\lambda x}x_1+x_1,\dots,(E_{\lambda x}x_i) \bullet x+x_i\bullet x, \dots,E_{\lambda x}x_n +x_n,x, \underbrace{e_{\lambda x},\dots,e_{\lambda x}}_{j\times} )
 \\&\quad\quad\quad\quad +
 \sum_{j\geq 1} \frac j {j!} m\circ (E_{\lambda x}x_1,\dots ,E_{\lambda x}x_n , e_{\lambda x}\bullet x +x, \underbrace{e_{\lambda x},\dots,e_{\lambda x}}_{j-1\times} )
 \\&= C\circ x.
\end{align*}
Hence the result follows.
\end{proof}

\subsection{Twisting \texorpdfstring{$\RT$}{RT}}
A Maurer-Cartan element in a pre-Lie algebra $\fg$ is a degree -1 element $\alpha\in \fg$ such that
\[
d\alpha + \alpha\bullet \alpha = 0.
\]
Of course, this is the same as a Maurer-Cartan element in $\fg$, considered as a dg Lie algebra.

We can apply the formalism of operadic twisting \cite{DolWill} to the operad $\RT$ (or rather, to map $\Lie \to \RT$) to produce another operad $\Tw\RT$ with a map $\Lie\to \Tw\RT$.
Concretely, $\Tw\RT$ can be seen as the operad generated by $\RT$ and an additional zero-ary generator of degree -1, completed in the number of such generators occurring.
The zero-ary generator is a formal avatar of the Maurer-Cartan element, and the differential on $\Tw\RT$ is defined to reflect the Maurer-Cartan equation.
Concretely, we can depict elements of $\Tw\RT$ as series of trees, with some vertices colored black (i.e., filled by the zero-ary generator).
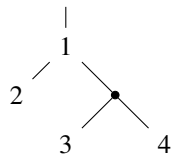
\[
\begin{tikzpicture}[scale=1.3]
\node (r) at (0,0) {$1$};
\node (v1) at (-.5,-.5) {$2$};
\node[int] (v2) at (.5,-.5) {};
\node (v3) at (0,-1) {$3$};
\node (v4) at (1,-1) {$4$};
\draw (r) edge +(0,.4) edge (v2) edge (v1)
          (v2) edge (v3) edge (v4);
\end{tikzpicture}
\]
The differential is then defined pictorially as follows.
\begin{align}
d
\begin{tikzpicture}[baseline=-.65ex]
\node (r) at (0,.3) {$j$};
\node (v1) at (-.7,-.7) {};
\node (v2) at (-.4,-.7) {};
\node at (.1,-.5) {$\scriptstyle \cdots$};
\node (v3) at (0.7,-.7) {};
\draw (r) edge +(0,.5) edge (v1) edge (v2) edge (v3);
\end{tikzpicture}
&=
\sum\,
\begin{tikzpicture}[baseline=-.65ex, yshift=-.3cm]
\node (rr) at (0,.8) {$j$};
\node[int] (r) at (0,.3) {};
\node (v1) at (-.7,-.5) {};
\node (v2) at (-.4,-.5) {};
\node at (.1,-.3) {$\scriptstyle \cdots$};
\node at (1.1,-0.3) {$\scriptstyle \cdots$};
\node (v3) at (0.7,-.5) {};
\draw (r) edge (v1) edge (v2) edge (v3)
(rr) edge (r)  edge +(0,.5) edge (1,-0.5) edge (1.5,-.5);
\end{tikzpicture}
-
\sum\,
\begin{tikzpicture}[baseline=-.65ex, yshift=-.3cm]
\node[int]  (rr) at (0,.8) {};
\node (r) at (0,.3) {$j$};
\node (v1) at (-.7,-.5) {};
\node (v2) at (-.4,-.5) {};
\node at (.1,-.3) {$\scriptstyle \cdots$};
\node at (1.1,-0.3) {$\scriptstyle \cdots$};
\node (v3) at (0.7,-.5) {};
\draw (r) edge (v1) edge (v2) edge (v3)
(rr) edge (r)  edge +(0,.5) edge (1,-0.5) edge (1.5,-.5);
\end{tikzpicture}
\\
\label{equ:gradiffint}
d
\begin{tikzpicture}[baseline=-.65ex]
\node[int] (r) at (0,.3) {};
\node (v1) at (-.7,-.7) {};
\node (v2) at (-.4,-.7) {};
\node at (.1,-.5) {$\scriptstyle \cdots$};
\node (v3) at (0.7,-.7) {};
\draw (r) edge +(0,.5) edge (v1) edge (v2) edge (v3);
\end{tikzpicture}
&=
\sum\,
\begin{tikzpicture}[baseline=-.65ex, yshift=-.3cm]
\node[int] (rr) at (0,.8) {};
\node[int] (r) at (0,.3) {};
\node (v1) at (-.7,-.5) {};
\node (v2) at (-.4,-.5) {};
\node at (.1,-.3) {$\scriptstyle \cdots$};
\node at (1.1,-0.3) {$\scriptstyle \cdots$};
\node (v3) at (0.7,-.5) {};
\draw (r) edge (v1) edge (v2) edge (v3)
(rr) edge (r)  edge +(0,.5) edge (1,-0.5) edge (1.5,-.5);
\end{tikzpicture}
.
\end{align}

For aesthetic reasons let us also define the suboperad
\[
\TRT \subset \Tw\RT 
\]
consisting of series of trees all of whose black vertices have at least two children.
It is finite dimensional in each arity.
The map 
\[
\Lie\to \TRT
\]
is defined by sending the generator to the combination of graphs \eqref{equ:mapfromlie}.
Similarly, we have the map 
\[
\TRT\to \RT 
\]
sending graphs with black vertices to zero.

\begin{thm}
The map $\Lie\to\TRT$ is a quasi-isomorphism of operads.
\end{thm}
\begin{proof}
The proof is a small adaptation of a trick by Lambrechts and Voli'c \cite{LV}.
We proceed by induction on the arity. Note that $\TRT(0)=0$ and $\TRT(1)\cong \K$, so in arity $\leq 1$ the statement is true. Now suppose it is true in arity $r-1$. We split (for $r\geq 2$)
\[
\TRT(r) = V_1 \oplus V_{\geq 2}
\]
where $V_1$ consists of series of trees with vertex 1 having valence 1, while $V_{\geq 2}$ consists of series of trees where vertex 1 has valence at least two.
The differential has pieces mapping $V_1$ to $V_1$, $V_{\geq 2}$ to $V_{\geq 2}$, and $V_{\geq 2}$ to $V_1$ by splitting off all incident edges.
We take a spectral sequence such that the first differential is only the last part $\delta': V_{\geq 2}\to V_1$, i.e., pictorially the map 
\[
\delta' : 
\begin{tikzpicture}[baseline=-.65ex]
\node (r) at (0,.3) {$j$};
\node (v1) at (-.7,-.7) {};
\node (v2) at (-.4,-.7) {};
\node at (.1,-.5) {$\scriptstyle \cdots$};
\node (v3) at (0.7,-.7) {};
\draw (r) edge +(0,.5) edge (v1) edge (v2) edge (v3);
\end{tikzpicture}
\mapsto 
\begin{tikzpicture}[baseline=-.65ex]
\node[int] (r) at (0,.3) {};
\node (j) at (1,-.3) {$j$};
\node (v1) at (-.7,-.7) {};
\node (v2) at (-.4,-.7) {};
\node at (.1,-.5) {$\scriptstyle \cdots$};
\node (v3) at (0.7,-.7) {};
\draw (r) edge +(0,.5) edge (v1) edge (v2) edge (v3) edge (j);
\end{tikzpicture}
\]
It is clear that the map is injective. The cokernel $V_1'=V_1/\mathit{im}(\delta')$ is spanned by graphs for which vertex 1 has valence 1 and is connected to another numbered vertex. It splits into $r-1$ components according to which numbered vertex the vertex 1 connects to. The spectral sequence abuts on the next page by degree reasons, and each of the $r-1$ components is itself isomorphic as a complex to $\TRT(r-1)$, with the isomorphism being given by removing vertex one and its adjacent edge. By the induction hypothesis we hence have
\[
H(\TRT(r))\cong \oplus_{j=1}^{r-1} \Lie(r-1).
\]
As a vector space this is the same as $\Lie(r)$, and we leave it to the reader to verify that the map $\Lie\to\TRT$ indeed induces an isomorphism in homology. 
\end{proof}

\begin{rem}\label{rem:exact}
We note in particular that the homology of $\TRT$ is concentrated in degree 0. Hence, any cocycle represented by graphs with $k$ black vertices for $k>0$ is necessarily exact.
\end{rem}

\subsection{A(nother) \texorpdfstring{$\hoLie_2$}{hoLie2}-structure }\label{sec:anotherLinfty}
Suppose that $\fg$ is a pre-Lie algebra and that $\alpha\in \fg$ is a Maurer-Cartan element.
Then we define a $\hoLie_2$-structure on $\fg$ such that the $r$-ary operation $\nu_r$ is, on arguments $x_1,\dots,x_r\in \fg$
\[
\nu_r(x_1,\dots,x_r)=
\begin{cases}
d x_1 + [\alpha,  x_1 ] & \text{for $r=1$} \\
\alpha \bullet (x_1,\dots,x_r) & \text{otherwise}
\end{cases}.
\]
Here we use the notation \eqref{equ:multibullet}.
Put differently, the above $\hoLie_2$ action on $\fg^\alpha$ is obtained via the $\TRT$-action on $\fg^\alpha$ through the map 
\[
\hoLie_2\to \TRT
\]
sending the generator $\nu_r$ to the graph
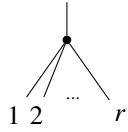
\[
\begin{tikzpicture}[baseline=-.65ex]
\node[int] (r) at (0,.3) {};
\node (v1) at (-.7,-.7) {$1$};
\node (v2) at (-.4,-.7) {$2$};
\node at (.1,-.5) {$\scriptstyle \cdots$};
\node (v3) at (0.7,-.7) {$r$};
\draw (r) edge +(0,.5) edge (v1) edge (v2) edge (v3);
\end{tikzpicture}
\]
Looking at the definition of the differential \eqref{equ:gradiffint} it is clear that the above map is indeed a map of operads and hence that the above formula for $\nu_r$ indeed defines a $\hoLie_2$-structure.

The main result here is that the above structure is trivial.
\begin{prop}\label{prop:Ltrivial}
The map $\hoLie_2\to \TRT$ is homotopic to the trivial map sending all generators $\nu_r$ for $r\geq 2$ to zero.
In particular, for any pre-Lie algebra $\alg g$ with Maurer-Cartan element $\alpha$ and $\nu_r$ as above there is an $L_\infty$-isomorphism
\[
(\fg^\alpha[1], 0) \to (\fg^\alpha[1], \nu_r) 
\]
where the left-hand side is understood as abelian dg Lie algebra.
\end{prop}
\begin{proof}
We have to construct a map 
\[
F: \hoLie_2\to \TRT[t,dt]
\]
interpolating from the trivial map at $t=0$ to the non-trivial one at $t=1$. 
We write $F= f_t +dt h_t$.
For the $0$-form part $f_t$ we just rescale $\nu_r$ above by $t^{r-1}$,
\[
f_t(\nu_r) = t^{r-1} f_1(\nu_r).
\] 
Now $h_t$ has to satisfy 
\[
d_{\TRT} h_t(\nu_r)  = -\partial_t f_t(\nu_r) - h_t(d_{\hoLie_2}\nu_r).
\]
This can be solved inductively on the arity $r$, noting that the right-hand side is closed and hence exact by Remark \ref{rem:exact}.
\end{proof}

While the above proof is very short, it is not constructive. Let us also provide a very simple explicit formula for the $L_\infty$-morphism.
\begin{prop}\label{prop:Ltrivialexpl}
An $L_\infty$-isomorphism
\[
W \colon (\fg^\alpha[1], 0) \to (\fg^\alpha[1], \nu_r) 
\]
is provided by the explicit formulas
\[
 W_r(x_1,\dots, x_r)
 =
 \frac 1 {r!} \sum_{\sigma\in S_r}
 \pm 
 (\cdots(x_{\sigma(1)}\bullet x_{\sigma(2)}) \bullet \cdots ) \bullet x_{\sigma(r)}  ,
\]
where $x_1,\dots,x_r\in \fg$ are homogeneous elements and the sign is the lexicographic one.
\end{prop}
\begin{proof}
Clearly $W_1$ is the identity map, so if the $L_\infty$-relations hold then $W$ is an $L_\infty$-isomorphism.
 The $L_\infty$-relations (to be shown) read
 \begin{multline*}
  \sum_{j=1}^r \pm W_r(x_1,\dots,dx_j,\dots, x_r)
  \\
  \stackrel{?}=
  d W_r(x_1,\dots, x_r)
  +
  \sum_{s=2}^r
  \frac{1}{s!}
  \sum_{j_1+\dots+j_s=r}
  \frac 1 {j_1!\cdots j_s!}
  \sum_{\sigma\in S_r}
  \alpha \bullet ( W_{j_1}(x_{\sigma(1)},\dots),\dots, W_{j_s}(\dots , x_{\sigma(r)} )  ) .
 \end{multline*}
Both sides are symmetric multilinear in the $x_j$. Hence, by graded polarization it is enough to show the formula for a formal element $X$ of degree zero. Summing over all $r\geq 1$ with prefactor $\frac 1 {r!}$ it suffices to check that 
 \begin{multline*}
  \sum_{r\geq 1} \frac 1 {r!} \sum_{j=1}^r  W_r(X,\dots,dX,\dots, X)
  \\
  \stackrel{?}=
  \sum_{r\geq 1} \frac 1 {r!} d W_r(X,\dots, X)
  +
  \sum_{r\geq 1} \frac 1 {r!} \sum_{s=2}^r
  \frac{1}{s!}
  \sum_{j_1+\dots+j_s=r}
  \frac 1 {j_1!\cdots j_s!}
  \sum_{\sigma\in S_r}
  \alpha \bullet ( W_{j_1}(X,\dots,X),\dots, W_{j_s}(X,\dots , X )  ) .
 \end{multline*} 
Note that $dx=[\alpha,X]$ and hence, using Lemma \ref{lem:preliederivation} the left-hand side becomes 
\[
\sum_{r\geq 1} \frac 1 {r!} \sum_{j=1}^r  \left( (\cdots(\alpha\bullet X)\bullet \cdots \bullet X
 -
  (\cdots(X\bullet X)\bullet \cdots \bullet X)\bullet \alpha
 \right).
\]
Note that using the notation of Lemma \ref{lem:distributive} we may write the two terms as 
\[
E_X\alpha -  e_X \bullet \alpha
\]
where we understood $\alg g$ as a pre-Lie module over $\alg g$.
Using the same notation, the right-hand side of the $L_\infty$-relations may be rewritten as
 \[
 [\alpha, e_X]
  +
  \sum_{s\geq 2}
  \frac{1}{s!}
  \alpha \bullet ( \underbrace{e_X.\dots,e_X}_{s\times} ) 
  =
  -e_X\bullet \alpha
  +
  \sum_{s\geq 1}
  \frac{1}{s!}
  \alpha \bullet ( \underbrace{e_X.\dots,e_X}_{s\times} ) 
  .
 \]
Using Lemma \ref{lem:distributive} the $L_\infty$-relations hence follow immediately.
\end{proof}

\begin{ex}
Let us note in particular that the Lie bracket $\nu_2$ is rendered exact by the homotopy 
\[
\frac 1 2\left(
\begin{tikzpicture}[baseline=-.65ex]
\node (v) at (0,.3) {$1$};
\node (w) at (0,-.3) {$2$};
\draw (v) edge (w) edge +(0,.5);
\end{tikzpicture} 
+
\begin{tikzpicture}[baseline=-.65ex]
\node (v) at (0,.3) {$2$};
\node (w) at (0,-.3) {$1$};
\draw (v) edge (w) edge +(0,.5);
\end{tikzpicture} 
\right).
\]
\end{ex}

\subsection{Twisting of (pre-)Lie modules}\label{sec:twistingmodules}
Now suppose that $\fg$ is a pre-Lie algebra and $M$ is a right $\fg$-module.
Let $\alpha\in \fg$ be a Maurer-Cartan element.
Then the action of $\alpha$ induces a differential $d+(-\circ \alpha)$ on $M$. Denote the dg vector space $M$ equipped with this differential by $M^\alpha$.
Then the twisted Lie algebra $\fg^\alpha$ acts from the right on $M^\alpha$.

Now suppose that $M$ in addition carries an action of the operad $\hoLie$, compatible with the right $\fg$-action in the sense of definition \ref{def:prelie}.
Then the twisted differential naturally also respects the $\hoLie$-structure, so we have a $\hoLie$ structure an $M^\alpha$ compatible with the right $\fg^\alpha$-action.

Let now $m\in M^\alpha$ be a Maurer-Cartan element. (We suppose here that $M$ is either pro-nilpotent or has only finitely many non-vanishing $L_\infty$-operations, so that the Maurer-Cartan equation makes sense.)
Twisting by $m$ produces another $\hoLie$-algebra 
\[
M^{\alpha,m} := (M^\alpha)^m.
\]
We note that the right $\fg^\alpha$-action is in general not compatible with the twisted $\hoLie$-structure.
Rather, we have a map 
\begin{gather*}
U_1 : \fg[1] \to M^{\alpha,m}
\\ 
x\mapsto \alpha \circ x, 
\end{gather*}
such that for $m_1,\dots,m_r\in M$, $x\in \fg$ and $\mu_r$ the $r$-th (twisted) $L_\infty$-operation of $M^{\alpha,m}$
\[
\mu_r(m_1,\dots,m_r) \circ x=  \sum_{j=1}^r \pm \mu_r(m_1,\dots, m_j\circ x,\dots,m_r) 
+
\mu_{r+1}(m_1,\dots,m_r,U_1(x)). 
\]

Now our main results will be derived from the following general Theorem.
\begin{thm}\label{thm:main_alg}
For $\fg,M,m,\alpha$ as above the operations 
\[
U_r(x_1,\dots,x_2) = m \circ (x_1,\dots,x_r)
\]
define an $L_\infty$-morphism
\[
(\fg^\alpha[1], \nu_r) \to (M^{\alpha,m}, \mu_r).
\]
\end{thm}

\begin{proof}
The Maurer-Cartan equation for $m$ reads
\begin{align}\label{equ:mMC}
&dm + m \circ \alpha + \sum_{k\geq 2} \frac 1 {k!} \mu_k(\alpha,\dots,\alpha) = 0.
\end{align}
We apply the operation $(-\circ (x_1,\dots,x_r))$ to this equation, giving us
\begin{multline}\label{equ:temp1}
0= d(m\circ (x_1,\dots,x_r)) - \sum_{j=1}^r \pm m\circ (x_1,\dots, dx_j,\dots, x_r)
+
\sum_{S\subset [r]} \pm m \circ (\alpha\bullet (x_S), x_{[r]\setminus S})
\\
+
\sum_{k\geq 2} \frac 1 {k!} \sum_{S1\sqcup \dots\sqcup S_k=[r]} \mu_k(\alpha\circ x_{S_1},\dots,\alpha\circ x_{S_k}).
\end{multline}
The first term is 
\beq{equ:temp3}
d(m\circ (x_1,\dots,x_r)) = d U_r(x_1,\dots,x_r).
\eeq
The second term is 
\beq{equ:temp4}
\sum_{j=1}^r \pm m\circ (x_1,\dots, dx_j,\dots, x_r) = \sum_{j=1}^r \pm U_r(x_1,\dots, dx_j,\dots, x_r).
\eeq
In the third term one may distinguish empty $S$, $S$ with one element, and $S$ with $\geq 2$ elements and obtain
\begin{align}
\label{equ:temp5}
\sum_{S\subset [r]} \pm m \circ (\alpha\bullet (x_S), x_{[r]\setminus S})
=
U_{r+1}(\alpha, x_1,\dots,x_r)
+
\sum_{j=1}^r \pm U_r(x_1,\dots ,\alpha\bullet x_j,\dots,x_r)
+
\sum_{\substack{S\subset [r] \\ |S|\geq 2}} \pm U_{r-|S|+1}(\nu_{|S|}(x_S), x_{[r]\setminus S})).
\end{align}
Finally in the fourth term of \eqref{equ:temp1} we may separate $S_j$ which are empty and non-empty. If there are $i$ non-empty $S_j$, those can be placed into $i$ of $k$ slots in ${k\choose i}$ ways. Hence we obtain
\begin{align}
&\sum_{k\geq 2} \frac 1 {k!} \sum_{S_1\sqcup \dots\sqcup S_k=[r]} \mu_k(\alpha\circ x_{S_1},\dots,\alpha\circ x_{S_k})
\nonumber \\  &=
\sum_{k\geq 2} \sum_{i=1}^k \frac 1 {i!(k-i)!} \sum_{\substack{S1\sqcup \dots\sqcup S_i=[r]\\ S_l\neq \emptyset \forall l}} \mu_k(\alpha,\dots,\alpha, U_{|S_1|}( x_{S_1}),\dots,U_{|S_i|}x_{S_i})
\nonumber \\ 
&=
\sum_{k\geq 2}\frac 1 {(k-1)!} \mu_k(\alpha,\dots,\alpha, U_{r}( x_1,\dots,x_r)
+
\sum_{i\geq 2} \frac 1 {i!} \sum_{\substack{S1\sqcup \dots\sqcup S_i=[r]\\ S_l\neq \emptyset \forall l}} \mu_k^\alpha(U_{|S_1|}( x_{S_1}),\dots,U_{|S_i|}x_{S_i})
. \label{equ:temp2}
\end{align}
Note that the first term of \eqref{equ:temp2} together with \eqref{equ:temp3} produces
\[
d^{\alpha,m} U_r(x_1,\dots,x_r) -   U_r(x_1,\dots,x_r)\circ \alpha.
\]
The second term of this equation in term is 
\begin{align}
&U_r(x_1,\dots,x_r)\circ \alpha
\nonumber \\&=
(m\circ (x_1,\dots,x_r))\circ \alpha
=
\sum_{j=1}^r \pm  m\circ (x_1,\dots,x_j\bullet \alpha,\dots, x_r)
+
m\circ (\alpha, x_1,\dots,x_j\bullet \alpha,\dots, x_r)
\nonumber
\\
&=
\label{equ:temp6}
\sum_{j=1}^r \pm  U_r (x_1,\dots,x_j\bullet \alpha,\dots, x_r)
+
U_{r+1} (\alpha, x_1,\dots,x_j\bullet \alpha,\dots, x_r)
\end{align}
The first term together with \eqref{equ:temp4} and the second term of \eqref{equ:temp5} gives
\[
\sum_{j=1}^r \pm U_r (x_1,\dots,d^m x_j,\dots, x_r).
\]
The second term of  \eqref{equ:temp6} kills the first term of \eqref{equ:temp5}.
Wrapping up all remaining terms we find
\begin{multline*}
0 =
d^{\alpha,m} U_r(x_1,\dots,x_r)
-
\sum_{j=1}^r \pm U_r (x_1,\dots,d^m x_j,\dots, x_r)
\\+
\sum_{i\geq 2} \frac 1 {i!} \sum_{\substack{S1\sqcup \dots\sqcup S_i=[r]\\ S_l\neq \emptyset \forall l}} \mu_k^\alpha(U_{|S_1|}( x_{S_1}),\dots,U_{|S_i|}x_{S_i})
-
\sum_{\substack{S\subset [r] \\ |S|\geq 2}} \pm U_{r-|S|+1}(\nu_{|S|}(x_S), x_{[r]\setminus S})).
\end{multline*}
This is precisely the $L_\infty$-relation.
\end{proof}

\subsection{A slight extension: Incorporating a derivation}
Now suppose we are in the situation of Theorem \ref{thm:main_alg}, but we have an additional piece of algebraic structure: A degree zero derivation $D$ on the $L_\infty$-algebra $M$, and a derivation also denoted by $D$ on $\fg$ which acts compatibly with the right $\fg$ action on $M$. This means that for $m_1\in M$, $x\in \fg$ we have
\[
D(m_1 \circ x) = (Dm_1) \circ x + m_1\circ (Dx).
\]
For simplicity we will also assume that $D$ annihilates our Maurer-Cartan element $\alpha\in \fg$, i.e.,
\[
D\alpha = 0.
\]
Then we have the following slight generalization.

\begin{thm}\label{thm:main_alg2}
For $\fg,M,m,\alpha,D$ as above the operations 
\begin{align*}
U_r(x_1,\dots,x_r) &= m \circ (x_1,\dots,x_r)
\\
U_{r+s}(\underbrace{D,\dots,D}_{s\times },x_1,\dots,x_r) &= (D^s m) \circ (x_1,\dots,x_r)
\end{align*}
define an $L_\infty$-morphism
\[
D \K[1] \oplus (\fg^\alpha[1], \nu_r) \to (M^{\alpha,m}, \mu_r).
\]
To be clear, the direct sum on the left-hand side is a direct sum of $L_\infty$-algebras, i.e., the element $D$ is central.  
\end{thm}
\begin{proof}
We consider the Maurer-Cartan equation for $m$, i.e., \eqref{equ:mMC}.
We first apply $D^s$ to this equation, and then compose with $(- \circ (x_1,\dots,x_r))$.
Then we simplify exactly as in the proof of Theorem \ref{thm:main_alg}. The net effect of pre-applying the operation $D^s$ is merely that summands with $k$ copies of $m$
\[
\underbrace{\dots m \dots m \dots m}_{\text{$k$ many $m$'s}}
\]
are replaced by the exact same summands, but with each $m$ precomposed by some $D^j$, so that the terms become
\[
\sum_{j_1+j_2+\dots+j_k =s}\frac {s!}{j_1!\cdots j_k!}
\underbrace{\dots (D^{j_1}m) \dots (D^{j_2}m) \dots (D^{j_k}m)}_{\text{$k$ many $m$'s}}.
\]
Making these replacement one arrives at the $L_\infty$-condition for the $L_\infty$-morphism defined in the Theorem.
\end{proof}

\section{Actions on and maps of Maurer-Cartan elements}\label{sec:GactionsonMC}
We continue the discussion of a pre-Lie algebra $\alg g$ with Maurer-Cartan element $\alpha\in \alg g$, a pre-Lie action on the $L_\infty$-algebra $M$, and a Maurer-Cartan element $m\in M$.
Using the $L_\infty$-morphisms of Proposition \ref{prop:Ltrivialexpl} and \ref{thm:main_alg} we obtain a chain of $L_\infty$-morphism 
\[
 (\alg g^\alpha[1],0) \xrightarrow{W} (\alg g^\alpha[1],\nu_r) \xrightarrow{U} (M^{\alpha,m},\mu_r).
\]
Such $L_\infty$-morphisms induce maps of Maurer-Cartan elements on the $L_\infty$-algebras involved, which we will discuss in this section.
In order to settle certain technical (but in practice non-critical) convergence issues, let us make the following assumption:

\medskip

{\bf Filtration assumption:} We assume that there are descending complete filtrations 
\begin{align*}
M &= \mF^0M \supset \mF^1 M\supset \cdots
&
\alg g &= \mF^0\alg g \supset \mF^1 \alg g \supset \cdots
\end{align*}
so that the pre-Lie structure on $\alg g$, the pre-Lie action on $M$ and the $L_\infty$-structure on $M$ are (additively) compatible with the filtrations.

\medskip

Since the $L_\infty$-morphisms above are built using only the natural algebraic structures they are also compatible with the filtrations.
Now, using the above filtrations we define the sets of Maurer-Cartan elements as elements of $\mF^1$ satisfying the Maurer-Cartan equation. Concretely, this means that $\MC( \alg g^\alpha[1],0)$ is just the set of degree 0 cocycles in $\mF^1\alg g^\alpha$,
while
\[
\MC(\alg g^\alpha[1], \nu_r)
=
\{\beta \in \mF^1 \alg g^\alpha[1] ; |\beta|=1, \sum_{r\geq 1}\frac 1 {r!} \nu_r(\beta,\dots,\beta)=0
\}
\]
and $\MC(M^{\alpha,m})$ is the set of degree 1 elements $m'\in \mF^1M$ satisfying 
\[
 \sum_{r\geq 1} \frac 1 {r!} \mu_r(m',\dots,m')=0,
\]
where the convergence of the series is guaranteed by the completeness of the filtration $\mF$, together with our assumption that the Maurer-Cartan element lives in $\mF^1$ in each case.

Now, by standard results of homological algebra, $L_\infty$-morphisms compatible with the filtrations (such as $W,U$ above) induce maps of the sets of Maurer-Cartan elements
\[
 \MC(\alg g^\alpha[1],0) \to \MC(\alg g^\alpha[1],\nu_r) \to \MC(M^{\alpha,m}),
\]
where, e.g., $\beta\in \MC(\alg g^\alpha[1],0)$ is sent to 
\[
 \sum_{r\geq 1} \frac 1 {r!} W_r(\beta,\dots, \beta).
\]
Again the sum converges because of the completeness of the filtration.

\begin{rem}
 The reader should consider the filtration above not as an essential piece of data but just as a technical nuissance whose presence we have to assert to ensure that some infinite series are converging. In practice it will eventually not play any essential role.
\end{rem}

\subsection{Actions of \texorpdfstring{$\alg g$}{g} on Maurer-Cartan elements}
Let $\alg g^\alpha_{0,\closed}\subset \alg g^\alpha$ be the Lie subalgebra of degree 0 cocycles.
The exponential group $\Exp(\mF^1\alg g^\alpha_{0,\closed})$ acts on $\MC(M^{\alpha,m},0)$ (which can be identified with $\mF^1\alg g^\alpha_{0,\closed}$) by the Baker-Campbell-Hausdorff-formula
\[
 \beta \cdot \exp(x) = \BCH(\beta, x) = \beta+x+\frac 1 2[\beta,x] + \cdots,
\]
where $\beta \in \MC(M^{\alpha,m})\cong \mF^1\alg g^\alpha_{0,\closed}$ and $x\in\mF^1\alg g^\alpha_{0,\closed}$.
The following result is obvious.
\begin{lemma}\label{lem:simplefreetransitive0}
 The action of $\Exp(\mF^1\alg g^\alpha_{0,\closed})$ on $\MC(\alg g^\alpha[1],0)$ is free and transitive.
\end{lemma}

Similarly, $\mF^1\alg g^\alpha_{0,\closed}$ acts on $\MC(\alg g^\alpha[1],\nu_r)$ by the formula
\[
 \beta \cdot x = x + \beta \bullet x,
\]
for $\beta \in \MC(g^\alpha[1],\nu_r)$ and $x\in\mF^1\alg g^\alpha_{0,\closed}$
This action integrates to an action of the exponential group $\Exp(\mF^1\alg g^\alpha_{0,\closed})$ on $\MC(g^\alpha[1])$ such that 
\beq{equ:gactiononMCg}
 \beta \cdot \exp(x) 
 :=
 \beta +
 \sum_{j\geq 1} \frac 1 {j!}\left( \underbrace{ (\cdots (\beta \bullet x)\cdots )\bullet x)}_{j\times} 
 +
 \underbrace{ x \bullet \cdots \bullet x}_{j\times} 
 \right)
\eeq

The following Lemma is relatively simple to show directly. (We leave it to the reader.)

\begin{lemma}\label{lem:simplefreetransitive}
 The action of $\Exp(\mF^1\alg g^\alpha_{0,\closed})$ on $\MC(\alg g^\alpha[1], \nu_r)$ is free and transitive.
\end{lemma}

The Lie algebra $\mF^1\alg g^\alpha_{0,\closed}$ also acts on $\MC(M^{\alpha,m})$ by the formula 
\[
 \beta \cdot x  = m\bullet x + \beta \bullet x,
\]
for $\beta \in \MC(M^{\alpha,m})$ and $x\in\mF^1\alg g^\alpha_{0,\closed}$.
Again this action integrates to an action of the exponential group $\Exp(\mF^1\alg g^\alpha_{0,\closed})$ on $\MC(M^{\alpha,m})$.
More concretely, $\exp(x)$ sends the Maurer-Cartan element $m'$ to
\[
 m' \cdot \exp(x) 
 :=
 m' + \sum_{j\geq 1} \frac 1 {j!} \underbrace{ (\cdots ((m'+m) \bullet x)\cdots )\bullet x)}_{j\times}.
\]

\subsection{Compatibility of the map with the actions on Maurer-Cartan elements}
Now we have three spaces of Maurer-Cartan elements related to each other by the maps $W$ and $U$ from above
\beq{equ:MCchain}
 \MC(\alg g^\alpha,0) \to \MC(\alg g^\alpha,\nu_r) \to \MC(M^{\alpha,m}).
\eeq
On each of these spaces the group $\Exp(\mF^1\alg g^\alpha_{0,\closed})$ acts.
We have the following result.
\begin{prop}\label{prop:equivariantMCmaps}
 The two maps of the chain \eqref{equ:MCchain} are equivariant with respect to the action of the group $\Exp(\mF^1\alg g^\alpha_{0,\closed})$.
\end{prop}
\begin{proof}
The proof is a straightforward calculation, given that we provided explicit formulas for all maps and actions. 
For example, the first map sends a Maurer-Cartan element $\beta \in  \MC(\alg g^\alpha,0)$ to 
\[
 e_\beta \in \MC(\alg g^\alpha,\nu_r)
\]
using the notation of Lemma \ref{lem:distributive}.
Acting by $\exp(x)$ on the image we obtain 
\[
 e_\beta\cdot \exp(x) = \exp((-)\bullet x) e_\beta + e_x.
\]
Using the BCH formula and noting that $e_\beta=\exp_\bullet(\beta)-1$ it is not hard to see that 
\[
\exp((-)\bullet x) e_\beta = e_{\BCH(\beta,x)} - e_x,
\]
so that the first half of the Proposition follows.

The second map of the chain \eqref{equ:MCchain} (which is induced by $U$) sends a Maurer-Cartan element $\beta\in\MC(\alg g^\alpha,\nu_r)$ to 
\[
 \sum_{r\geq 1} \frac{1}{r!} m\circ (\beta,\dots,\beta ).
\]
Acting with $\exp(x)$ we obtain 
\[
(E_x+\mathit{id}) \left(\sum_{r\geq 1} \frac{1}{r!} m\circ (\beta,\dots,\beta )\right)
 +E_x m.
\]
By Lemma \ref{lem:distributive} this can be rewritten as
\begin{align*}
 &
 \sum_{r\geq 0}\sum_{j\geq 0} \frac{1}{r!j!}
 m\circ (E_x\beta+\beta,\dots,E_x\beta+\beta,  e_x,\dots,e_x) -m
 \\
 &=
 \sum_{r\geq 1}\frac{1}{r!} m\circ (E_x\beta+\beta+e_x,\dots,E_x\beta+\beta+ e_x),
\end{align*}
which in turn is precisely the image of 
\[
 \beta\cdot \exp(x) = E_x \beta + \beta + e_x
\]
under the second map of map of \eqref{equ:MCchain}.
\end{proof}

\subsection{Torsor property}
On the sets of Maurer-Cartan elements $\MC(\alg h)$ of a pro-nilpotent $L_\infty$-algebra $\alg h$ one has an action of the exponential group of the degree zero subalgebra $\Exp(\alg h_0)$ by gauge transformations.
Concretely, given an element $\beta \in \MC(\alg h)$ and $x\in \alg h_0$ the infinitesimal version of this gauge action is defined as
\[
 \beta \mapsto \beta \cdot x = d_{\beta} x := \sum_{r\geq 0} \frac{1}{r!} \mu_{r+1}(\beta,\dots,\beta,x).
\]
The action of the exponential group is obtained by integrating this formula. For example, in the case of $\alg h$ a dg Lie algebra one obtains 
\[
 \beta \mapsto \beta \cdot \gexp(x) = \frac{e^{\ad_x} -1 }{\ad_x} dx + e^{\ad_x}\beta,
\]
where we introduced the notation $\gexp(x)$ to denote the exponential in the gauge group, and $\ad_x(-):=[-,x]$.
In the yet more special case that $\alg h$ is abelian, the gauge axtion of $\gexp(x)$ is merely the addition of $dx$.
In general, denote the set of gauge equivalence classes of Maurer-Cartan elements by 
\[
 \oMC(\alg h).
\]

As a Corollary of Proposition \ref{prop:equivariantMCmaps} we now find the following result, which we will use below.

\begin{cor}\label{cor:torsors}
 Let $\alg h\subset \mF^1\alg g^h$ be a dg Lie subalgebra. Suppose that $N\subset M^{\alpha,m}$ is an $L_\infty$ subalgebra such that the restriction of $U\circ W$ to $\alg h$ is an $L_\infty$ morphism taking values in $N$
 \[
  V: (\alg h[1],0) \to N,
 \]
where the notation shall indicate that we regard the left-hand side as an abelian $L_\infty$-algebra.
Suppose furthermore that there are complete descending filtrations on $\alg h$ and $N$
\begin{align*}
M &= \mF^0M \supset \mF^1 M\supset \cdots
&
\alg g &= \mF^0\alg g \supset \mF^1 \alg g \supset \cdots
\end{align*}
compatible with all structures and $V$, such that $V$ induces a quasi-isomorphism on the associated graded complexes.
Then $V$ induces an isomorphism 
\[
 \oMC(\alg h[1],0) \to \oMC(N)
\]
and both sides are $\Exp(H^0(\alg h))$-torsors, with the map respecting the action.
\end{cor}
\begin{proof}
 By the Goldmann-Millson Theorem \cite{DolRog} the map $V$ induces an isomorphism $\oMC(\alg h[1],0) \to \oMC(N)$. Furthermore, by Proposition \ref{prop:equivariantMCmaps} this isomorphism respects the action of $\Exp(H^0(\alg h))$. Hence one side is a torsor iff the other is as well. But it is clear that $\oMC(\alg h[1],0)\cong H^0(\alg h)$ is an $\Exp(H^0(\alg h))$-torsor.
\end{proof}

\section{Graph operads and graph complexes}
We will use the graph complexes $\GC_n$ and the graph operads $\Graphs_n$ both originally defined by M. Kontsevich \cite{K2,Kformal}. For the versions we use here we refer the reader to recollection of \cite[section 7]{FTW}.
Concretely, elements of $\GC_n$ are (possibly infinite) $\K$-linear combinations of isomorphism classes or undirected, connected graphs with at least trivalent vertices, for example
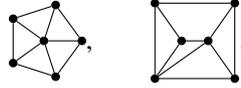
\[
    \begin{tikzpicture}[baseline=-.65ex, scale=.5]
\node[int] (c) at (0,0){};
\node[int] (v1) at (0:1) {};
\node[int] (v2) at (72:1) {};
\node[int] (v3) at (144:1) {};
\node[int] (v4) at (216:1) {};
\node[int] (v5) at (-72:1) {};
\draw (v1) edge (v2) edge (v5) (v3) edge (v2) edge (v4) (v4) edge (v5)
      (c) edge (v1) edge (v2) edge (v3) edge (v4) (c) edge (v5);
\end{tikzpicture}
,
\quad\quad
    \begin{tikzpicture}[baseline=-.65ex, scale=.5]
\node[int] (c) at (0.7,0){};
\node[int] (v1) at (0,-1) {};
\node[int] (v2) at (0,1) {};
\node[int] (v3) at (2.1,-1) {};
\node[int] (v4) at (2.1,1) {};
\node[int] (d) at (1.4,0) {};
\draw (v1) edge (v2) edge (v3)  edge (d) edge (c) (v2) edge (v4) edge (c) (v4) edge (d) edge (v3) (v3) edge (d) (c) edge (d);
\end{tikzpicture}.
\]
If we also allow bivalent vertices we obtain a similar complex $\GC_n^2\supset \GC_n$.
The important fact for us is that the graph complexes $\GC_n^2$ and $\GC_n$ carry a pre-Lie algebra structure.
The pre-Lie product $\Gamma\bullet \Gamma'$ is obtained by summing over all ways of inserting $\Gamma'$ into a vertex of $\Gamma$ and reconnecting the "dangling" edges.

The $r$-ary operations of the graph operads $\Graphs_n$ are (possibly infinite) linear combinations of graphs with $r$ numbered ("external") and an arbitrary (but finite) number of unidentifiable ("internal") vertices, e.g., 
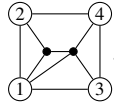
\[
  \begin{tikzpicture}[baseline=-.65ex, scale=.5]
\node[int] (c) at (0.7,0){};
\node[ext] (v1) at (0,-1) {$\scriptstyle 1$};
\node[ext] (v2) at (0,1) {$\scriptstyle 2$};
\node[ext] (v3) at (2.1,-1) {$\scriptstyle 3$};
\node[ext] (v4) at (2.1,1) {$\scriptstyle 4$};
\node[int] (d) at (1.4,0) {};
\draw (v1) edge (v2) edge (v3)  edge (d) edge (c) (v2) edge (v4) edge (c) (v4) edge (d) edge (v3) (v3) edge (d) (c) edge (d);
\end{tikzpicture}.
\]
The operadic composition is defined by insertion of a graph into an external vertex of another.
The important point for us is that the graded operad $\Graphs_n$ (i.e., with zero differential) is a pre-Lie module over the graded pre-Lie algebra $\GC_n$. 
For graphs $\Gamma\in \Graphs_n$, $\gamma\in \GC_n$ 
the right action $\Gamma\bullet \gamma$ is defined by summing over all ways of inserting $\gamma$ into an internal vertex of $\Gamma$ and reconnecting the dangling edges. We refer to \cite{DolWill} for the definition of this operation in the proper generality.

\subsection{The hairy graph complexes}\label{sec:HGC}
It is well known that the total invariant space of any dg operad carries a natural dg Lie algebra structure.
We apply this to the the degree shifted operad $\Graphs_n\{m\}$ to conclude that the space
\[
\prod_{r\geq 1} \Graphs_n\{m\}^{\bbS_r}
\]
is a dg Lie algebra. Consider the subspace
\[
\fHGC_{m,n}\subset \prod_{r\geq 1} \Graphs_n\{m\}^{\bbS_r}
\]
consisting of graphs all of whose external vertices have valence exactly $1$.
One can check that $\fHGC_{m,n}$ is in fact a dg Lie subalgebra, which we call the full hairy graph complex.
Pictorially, elements of $\fHGC_{m,n}$ are (possibly infinite) linear combinations of graphs with hairs, as depicted in \eqref{equ:hairysample}. In those pictures, the univalent external vertices sitting at the ends of the graphs are not drawn.

Now, from the right action of $\GC_n$ on $\Graphs_n$ we obtain a right action of the graded pre-Lie algebra $\GC_n$ on the graded Lie algebra $\HGC_{m,n}$.

Similarly, in the above definitions we may allow bivalent internal vertices and obtain a version of the hairy graph complex $\HGC_{m,n}^2$, or we even allow univalent internal vertices to obtain a version of the hairy graph complex $\HGC_{m,n}^1$. We obtain a right pre-Lie action of $\GC_n^2$ on $\HGC_{m,n}^2$, and a right pre-Lie action of $\GC_n^1$ on $\HGC_{m,n}^1$, compatible with the graded Lie algebra structure.

Summarizing, we are in the following situation, which is a special case of the twisting construction from section \ref{sec:twistingmodules}.
\begin{itemize}
\item We have a graded pre-Lie algebra $\GC_n^1$ together with a Maurer-Cartan element 
\beq{equ:MCalpha}
\alpha = 
\begin{tikzpicture}[baseline=-.65ex]
\draw (0,0) node[int] {} -- (.5,0) node[int] {};
\end{tikzpicture}.
\eeq
Twisting by this Maurer-Cartan element we obtain the graph complex $\GC_n^1$.
\item We have a graded Lie algebra $\HGC_{m,n}^1$ together with a compatible right $\GC_n^1$-action.
\item We have an element 
\beq{equ:smallm}
m
=
\begin{tikzpicture}[baseline=-.65ex]
\node[int] (v) at (0,.3){};
\draw (v) edge (0,-.3);
\end{tikzpicture}\,
\in \HGC_{m,n}^1, 
\eeq
which is a Maurer-Cartan element in the $\alpha$-twisted dg Lie algebra $\HGC_{m,n}^1$.
\item The hairy graph complex with the "correct" differential is obtained by twisting by $\alpha$ and $m$, i.e., for $\Gamma\in \HGC_{m,n}^1$ we have
\[
d \Gamma = [m,\Gamma] + (-1)^{|\Gamma|} \Gamma \bullet \alpha.
\]
\end{itemize}

The operations above are such that the sub-graded vector spaces $\HGC_{m,n}\subset \HGC_{m,n}^2\subset \HGC_{m,n}^1$ are preserved by the differential, although $m$ is not an element in the smaller complexes. (All terms in $d$ producing terms of valence 1 or 2 cancel as long as no such vertices were present before, as a small calculation shows.)

\subsection{Recollection from \cite{WillDefQ}}
We recall from \cite{WillDefQ} that the hairy graph complexes $\HGC_{1,n}$ are naturally equipped with a nontrivial $L_\infty$-structure, called in loc. cit. the Shoikhet $L_\infty$-structure after \cite{Shoikhet}.
Let us briefly recall how this $L_\infty$-structure is constructed.
\begin{itemize}
\item There is an $L_\infty$-structure on the graded space $\HGC_{m,n}^1$, compatible with the right $\GC_n^1$ action.
We call this structure the \emph{pre-Shoikhet} $L_\infty$-structure.
\item The Maurer-Cartan element $m$ from \eqref{equ:smallm} above is also a Maurer-Cartan element with respect to the pre-Shoikhet $L_\infty$-structure, twisted by $\alpha$.
\item Twisting by $m$ and $\alpha$ recovers an $L_\infty$-structure (the \emph{Shoikhet}-$L_\infty$-structure) whose differential is the standard differential on the hairy graph complex $\HGC_{1,n}^1$.
\item The $L_\infty$-structure restricts to the subspaces $\HGC_{1,n}\subset \HGC_{1,n}^2\subset \HGC_{1,n}^1$.
\item The $L_\infty$-structure respects the complete grading by the number of edges minus the number of internal vertices of graphs.
\end{itemize}

Again, the $L_\infty$ structure on $\HGC_{1,n}^1$ is produced by twisting the $L_\infty$ structure on the underlying graded space by $\alpha$ and $m$ (as in section \ref{sec:twistingmodules}), and in particular Theorem \ref{thm:main_alg} is applicable.

\section{Proofs of the main Theorem \ref{thm:main}} 
We now apply Theorem \ref{thm:main_alg2} to the following special cases:
\begin{itemize}
\item $\fg=\GC_n^1$.
\item $M=\HGC_{m,n}^{1,\circ}$, where the $\circ$ shall mean that we disregard the differential.
\item The MC element $\alpha\in \fg$ is always \eqref{equ:MCalpha}.
\item The $L_\infty$-structure on $M$ is either the canonical graded Lie algebra structure, or the pre-Shoikhet structure (if $m=1$, $n=2$).
\item The Maurer-Cartan element $m$ is \eqref{equ:smallm}, plus one of $L,T,T'$ occurring in the statement of Theorem \ref{thm:main_alg2}.
\item The derivation $D$ is defined on $\HGC_{m,n}^1$ as the grading generator wrt. the grading by number of edges minus number of internal vertices, and on $\GC_n^1$ the operation $D$ is defined to be the generator of the loop order grading.
\end{itemize}
Note that by these choices the twisted $L_\infty$-structure on $M^{\alpha,m}$ is in all three cases "the correct" one, i.e., either the canonical dg Lie algebra structure or the Shoikhet $L_\infty$-structure.

Now, we obtain a $\hoLie_2$-structure on $\GC_n^1$ which is in all cases a dg Lie algebra structure, whose differential is the canonical one, and whose bracket is given by connecting two graphs by one edge.
Applying Theorem \ref{thm:main_alg2} we also obtain an $L_\infty$-morphism
\[
D\K[1] \oplus (\GC_n^1[1],\nu) \to \HGC_{m,n}^1.
\]
Furthermore, by looking at the definition of both the $L_\infty$-structures and the $L_\infty$-morphism one sees that the operations cannot create univalent vertices if there were none before.
Hence all structure restricts to the $\geq $bivalent pieces, and in particular there is an $L_\infty$-morphism
\[
D\K[1] \oplus (\GC_n^2[1],\nu) \to \HGC_{m,n}^2.
\]
Next we precompose $U$ with the $L_\infty$ morphism trivializing the Lie bracket on $\GC_n^2$, whose existence is asserted in Proposition \ref{prop:Ltrivial} to obtain the desired $L_\infty$-morphism
\[
 D\K[1] \oplus (\GC_n^2[1],0) \to \HGC_{m,n}^2.
\]
To finish the proof, one merely notes that the leading order piece of the $L_\infty$-morphism agrees with the maps of Theorems \ref{thm:pre1}-\ref{thm:pre3} in each case, and hence it follows that the above maps are $L_\infty$-quasi-isomorphisms.
\hfill\qed
  
\begin{rem}
 We stated our main Theorem \ref{thm:main} for the special Maurer-Cartan elements $L$, $T$ and $T'$. However, it is clear from the proof that the Theorem holds for more general Maurer-Cartan elements, as long as the coefficient of the line graph ($L$) is nonzero for $\HGC_{n,n}$, or as long as the coefficient of the tripod graph is nonzero for $\HGC_{n-1,n}$ and $\HGC_{1,2}'$.
\end{rem}
  
\section{Maurer-Cartan elements}
In this section we collect several observations about the sets of Maurer-Cartan elements in the hairy graph complexes, that can be derived from methods of this paper.

\subsection{Maurer-Cartan elements in \texorpdfstring{$\HGC_{n,n}$}{HGCnn} and \texorpdfstring{$\HGC_{n-1,n}$}{HGCn1n}}
The dg Lie algebra $\HGC_{n,n}$ (for $n\geq 2$) contains a very simple family of Maurer-Cartan elements, namely the graphs 
\[
\lambda\,
 \begin{tikzpicture}[baseline=-.65ex]
  \draw (0,0)--(1,0);
 \end{tikzpicture}
\]
for $\lambda\in \K$.

Similarly, one can verify that for each $\lambda \in \K$ the following series is a Maurer-Cartan element in $\HGC_{n-1,n}$.
\beq{equ:tripodMCla}
 \sum_{k\geq 1} \lambda^k 
 \underbrace{
 \begin{tikzpicture}[baseline=-.65ex]
\node[int] (v) at (0,0) {};
\draw (v) edge +(-.7,-.5)  
edge +(-.4,-.5) edge +(.4, -.5) 
edge +(.7,-.5);
\node at (0,-.3) {$\scriptstyle \cdots$};
\end{tikzpicture}
}_{2k+1 \times}
\eeq
In particular there are Maurer-Cartan elements whose coefficient in front of the tripod graph is any number in $\K$.
Most of the remainder of this section is devoted to showing the analogous statement for $\HGC_{1,2}'$, equipped with the Shoikhet $L_\infty$-structure, where it is significantly harder.
The (informal) reason for this difficulty is that Maurer-Cartan elements in $\HGC_{1,2}'$ with non-vanishing coefficient in front of the tripod graph are essentially equivalent data to a Drinfeld associator, and such objects are notoriously hard to construct.
In particular, all known construction need to use non-algebraic (transcendental) methods at some point.

\subsection{Existence of real Maurer-Cartan elements in \texorpdfstring{$\HGC_{1,2}'$}{HGC12'}} 
The existence of rational Maurer-Cartan elements in $\HGC_{1,2}'$ follows from results of \cite{FTW}. However, for clarity and self-containedness of the exposition let us sketch in this subsection and the next how such Maurer-Cartan elements may be constructed directly.

We will begin by constructing a Maurer-Cartan element in $\HGC_{1,2}'$ over $\K=\R$.
To this end we will first recall from \cite{TW} the construction of a canonical Maurer-Cartan element in the convolution dg Lie algebra
\[
 \Conv(\Ass^\vee, \Graphs_2).
\]
Concretely, such Maurer-Cartan elements are given by operad maps 
\[
 \Ass_\infty = \Omega(\Ass^\vee) \to \Graphs_2.
\]
There is a canonical such map given by the composition
\beq{equ:temp10}
 \Ass_\infty \to C(\FM_1) \to C(\FM_2) \to \Graphs_2.
\eeq
Here $\FM_d$ is the Fulton-MacPherson-Axelrod-Singer compactification of the configurationspace of points in $\R^d$ \cite{GJ}, and $C(\FM_d)$ denotes the operad of semi-algebraic chains thereof \cite{HLTV, LVformal}.
The first map in the chain \eqref{equ:temp10} sends the $r$-ary generating $A_\infty$ operation to the fundamental chain $F_r$ of the connected component of $\FM_1$ for which the points are in ascending order.
The second map in \eqref{equ:temp10} is induced by the natural map $\FM_1\to \FM_2$ obtained from the standard embedding $\R\to \R^2$. Finally, the third map in \eqref{equ:temp10} is defined by Kontsevich. It sends a chain $c\in C(\FM_2(r))$ to the series of graphs 
\[
 \sum_\Gamma \Gamma \int_c \omega_\Gamma
\]
where the sum is over graphs forming a basis of $\Graphs_2(r)$, and $\omega_\Gamma$ is a (PA- \cite{HLTV}) differential form on naturally associated to $\Gamma$, see \cite{K2,LVformal,TW} for more details.

Let us however isolate some explicit formula for the Maurer-Cartan element in $\Conv(\Ass^\vee, \Graphs_2)$. Concretely, we may identify (as dg graded vector space)
\[
 \Conv(\Ass^\vee, \Graphs_2)
 \cong 
 \prod_{r\geq 1} \Ass\{1\}(r) \otimes_{S_r} \Graphs_2(r)
 \cong 
 \prod_{r\geq 1} \Graphs_2(r)[1-r].
\]
The desired Maurer-Cartan element is then 
\beq{equ:MCintegral}
\mu = \sum_r \sum_\Gamma \Gamma \int_{F_r} \omega_\Gamma,
\eeq
where $F_r\in C(\FM_1(r))\subset C(\FM_2(r))$ is the fundamental chain of the connected component of $\FM_1(r)$ in which the $r$ points are in ascending order on the line.

The link to the hairy graph complexes is that (as shown in \cite{WillDefQ}) there is an $L_\infty$-quasi-morphism
\beq{equ:temp11}
U_{\rm Sh} \colon \fHGC_{1,2}' \to \Conv(\Ass^\vee, \Graphs_2)
\eeq
from the full (disconnected) hairy graph complex $\fHGC_{1,2}$ equipped with the Shoikhet $L_\infty$ struture. By the Goldman-Millson Theorem \cite{DolRog} we can hence transfer our Maurer-Cartan element to a Maurer-Cartan element in $\fHGC_{1,2}'$.

We next claim that this MC element lives (or rather: can be chosen to live) in the connected sub-$L_\infty$-algebra $\HGC_{1,2}'\subset \fHGC_{1,2}'$.
To this end we first define a ``connected'' sub-dg Lie algebra 
\[
 \Conv(\Ass^\vee,\Graphs_2)_\conn \subset \Conv(\Ass^\vee,\Graphs_2),
\]
following the lines of analogous definitions in \cite{WillSCModel}.
To this end we identify the collection $\Ass$ with the collection $\Poiss_1$.
We note explicitly that $\Ass\cong \Poiss$ as collections, but not as operads.
Now we can identify (as dg vector space) 
\[
\Conv(\Ass^\vee,\Graphs_2)
\cong 
 \prod_{r\geq 1} \Poiss_1\{1\}(r) \otimes_{S_r} \Graphs_2(r).
\]
Now, using that $\Poiss_1=\Com\circ \Lie$ we may depict elements of this space as graphs with internal and external vertices, whose external vertices may in addition be attached to Lie trees, as the following exemplary picture indicates.
\[
 \begin{tikzpicture}
  \node[ext] (v1) at (0,0) {};
  \node[ext] (v2) at (0.5,0) {};
  \node[ext] (v3) at (1,0) {};
  \node[ext] (v4) at (1.5,0) {};
  \node[ext] (v5) at (2,0) {};
  \node[int] (i1) at (0.5,.5) {};
  \coordinate (t1) at (0,-.5);
  \coordinate (t2) at (0.75,-.5);
  \coordinate (t3) at (1.25,-1);
  \coordinate (t4) at (1.25,-1.5);
 \coordinate (t5) at (2,-.5);
   \draw (i1) edge (v1) edge (v2) edge (v3) 
             (v4) edge[bend left] (v5);
   \draw[dotted] (t1) edge (v1) (t2) edge (v2) edge (v3) edge (t3) (t3) edge (t4) edge (v4) (t5) edge (v5); 
   \node (s1) at (4.8,.5) {$\Graphs_2(5)$ element};
   \node (s2) at (4.2,-.7) {Lie trees};
   \draw[-latex] (s1) edge (2.1,.5) (s2) edge (2.1,-.7);
 \end{tikzpicture}
\]
We now define 
\[
\Conv(\Ass^\vee,\Graphs_2)_\conn \subset \Conv(\Ass^\vee,\Graphs_2) 
\]
as the subspace spanned by connected such graphs, where we consider external vertices connected to the same Lie tree as in the same connected component. For example, the example graph shown in the picture above is connected. 
One may verify three statements:
\begin{itemize}
 \item $\Conv(\Ass^\vee,\Graphs_2)_\conn \subset \Conv(\Ass^\vee,\Graphs_2) $ is a sub-dg Lie algebra.
 \item The Maurer-Cartan element $\mu$ from above is contained in $\Conv(\Ass^\vee,\Graphs_2)_\conn$. (The proof is similar to \cite[Proposition 6]{WillSCModel}.)
 \item The restriction of the ``Shoikhet formality morphism'' $U_{Sh}$ of \eqref{equ:temp11} to $\HGC_{1,2}'\subset \fHGC_{1,2}'$ takes values in the connected component $\Conv(\Ass^\vee,\Graphs_2)_\conn$. This restriction is an $L_\infty$-quasi-isomorphism $\HGC_{1,2}'\to \Conv(\Ass^\vee,\Graphs_2)_\conn$.
 (Here one uses that the Shoikhet formality morphism is connected in the sense of \cite[Definition 9]{WillSCModel}.)
\end{itemize}

Given these verifications it follows again by the Goldmann-Millson Theorem that the Maurer-Cartan element $\mu\in \Conv(\Ass^\vee,\Graphs_2)_\conn$ can be transferred to a Maurer-Cartan element in $\HGC_{1,2}'$. Furthermore, by explicit evaluation of the integral (done in \cite{TW}) one finds that the coefficient of the tripod graph is $\frac 1 4$. 
 
We arrive at the following conclusion.

\begin{prop}
 Over $\K=\R$ a Maurer-Cartan element of $\HGC_{1,2}'$ exists.
\end{prop}

Furthermore, we make the following observations. FIrst, $\Graphs_2$ has a natural grading by the number 
\[
 \#(\text{edges})- \#(\text{internal vertices}).
\]
It induces a grading on $\Conv(\Ass^\vee,\Graphs_2)_\conn$. We call the grading generator by $G$. We call a Maurer-Cartan element $\nu$ \emph{even} if 
\[
(-1)^G \nu = \nu.
\]

\begin{lemma}\label{lem:evenMC}
The Maurer-Cartan element $\mu\in\Conv(\Ass^\vee,\Graphs_2)_\conn$ is even, i.e.,
\[
(-1)^G \mu = \mu.
\]
\end{lemma}
\begin{proof}
One has to check that odd graded graphs $\Gamma$ are assigned zero weight in \eqref{equ:MCintegral}, i.e.,
\[
\int_{F_r} \omega_\Gamma = 0.
\]
But this follows from a simple reflection argument:Reflecting the plane at the $x$-axis changes $\omega_\Gamma$ to $(-1)^{ \#(\text{edges})}\omega_\Gamma$ and changes the orientation of the configuration space by $(-1)^{ \#(\text{vertices})}$. Hence if the degree of $\Gamma$ is odd, the integral vanishes.
\end{proof}

\begin{cor}\label{cor:anylambdaMC}
There are Maurer-Cartan elements of $\Conv(\Ass^\vee,\Graphs_2)_\conn$ and $\HGC_{1,2}'$ whose tripod coefficient is any real number. 
\end{cor}
\begin{proof}
We note that the tripod coefficient in $\mu$ is $\frac 1 4$ as computed in \cite{TW}. Given any $\lambda\in \R$, rescaling by the grading generator produces a Maurer-Cartan element
\[
(\sqrt{\lambda})^G\mu
\]
produces a Maurer-Cartan element in $\Conv(\Ass^\vee,\Graphs_2)_\conn$ whose tripod coefficient is $\lambda$. Note that $\sqrt{\lambda}$ might not be in $\R$. However, the operation is still well defined since by Lemma \ref{lem:evenMC} only even degree graphs appear non-trivially in $\mu$.
By transferring this MC element to $\HGC_{1,2}' \simeq \Conv(\Ass^\vee,\Graphs_2)_\conn$ we obtain the desired (real) MC element in the hairy graph complex with tripod coefficient $\lambda$.
\end{proof}

\subsection{From real to rational MC elements}
Next, let us discuss how to ``rationalize'' a real MC element of $\HGC_{1,2}'$.
First, let us remark on a minor issue: The Shoikhet $L_\infty$-structure on $\HGC_{1,2}'$ is canonically defined only up to homotopy. However, we may always adjust this homotopy so that the Shoikhet $L_\infty$-structure is defined over $\Q$, and we may always transfer any MC element to an MC element with respect to this adjusted $L_\infty$-structure via the homotopy. (This process does not alter the tripod coefficient.)
We hence assume in this section that we have a MC element $\mu \in \HGC_{1,2,\R}'$, but the $L_\infty$-structure on $\HGC_{1,2,\R}'$ is already rational.

\begin{thm}
 Let $\mu\in \HGC_{1,2,\R}'$ be a Maurer-Cartan element with rational coefficient in front of the tripod graph. Then there is a Maurer-Cartan element $\nu\in \HGC_{1,2,\Q}'\subset \HGC_{1,2,\R}'$, such that the coefficients in front of the tripod graph of $\mu$ and $\nu$ coincide.
\end{thm}
The argument is modeled on the standard proof that rational Drinfeld associators exist.
\begin{proof}
We assume that the tripod coefficient is $\lambda\neq 0$, otherwise we can just take $\nu=0$ trivially.
 We filter $\HGC_{1,2,\R}'$ by the loop order.
 We will induct on the loop order and apply the gauge group action and the action of $\Exp(\GC_2)$ as in section \ref{sec:GactionsonMC} to eliminate irrational coefficients in front of graphs.
 First, in loop order zero one may check that by the Maurer-Cartan equation $\mu$ must necessarily agree with \eqref{equ:tripodMCla}. In particular, all coefficients in loop order 0 are already rational if $\lambda$ is.
 Next suppose that $\mu$ is rational up to and including loop order $k-1$. We want to modify $\mu$ to a different MC element $\mu'$ that is rational up to and including loop order $k$.
 Let $\mu_j$ be the piece of $\mu$ of loop order $j$ so that $\mu=\sum_{j\geq 0}\mu_j$.
 The Maurer-Cartan equation in loop order $k$ reads
 \[
  d\mu_k = (\text{terms depending on $\mu_0,\dots,\mu_{k-1}$}),
 \]
 and in particular the right-hand side is rational by the induction hypothesis.
 Not the space of solutions to the above rational linear equation is a rational affine space modeled on the kernel of $d$. This rational affine space is not empty since a real solution ($\mu_k$) exists. Hence a rational solution $\mu_k'$ exists, and we must have
 \[
  d (\mu_k -\mu_k')=0. 
 \]
 Now, by the developments of the previous sections we have a quasi-isomorphism $\phi_\mu:\GC_2\to (\HGC_{1,2}')^\mu$. Thus we may find $b_k\in \HGC_{1,2}$ of degree 0 and loop order $k$ and a closed $x\in \GC_2$ of degree 0 such that 
 \[
  \mu_k +db_k + (\phi_\mu(x))_k = \mu_k',
 \]
 where $(\phi_\mu(x))_k$ is the part of loop order $k$.
Now, we have an action of $\Exp(\GC_{2,0,\closed})$ on $\MC(\HGC_{1,2}')$ according to section \ref{sec:GactionsonMC}, and furthermore 
\[
 \mu_k  + (\phi_\mu(x))_k = (\mu\cdot \exp(x))_k,
\]
where the subscript $k$ on the right-hand side shall indicate taking the part of loop order $k$.
We now define the desired Maurer-Cartan element
\[
 \mu' = \mu\cdot \exp(x) \cdot \gexp(b_k)
\]
by applying the action of $\exp(x)$ followed by the gauge action of the exponential of $b_k$.
Note that since both $x$ and $b_k$ live in loop order $k$, and since the actions respect the loop order filtrations, the elements $\mu$ and $\mu'$ agree in loop orders $< k$. The loop order $k$ piece of $\mu'$ is the $\mu_k'$ above by construction and hence rational. Hence $\mu'$ is rational in loop orders $\leq k$ by construction. Proceeding in this way and taking the limit we find the desired rational Maurer-Cartan element $\nu$.
\end{proof}

Together with Corollary \ref{cor:anylambdaMC} we immediately obtain the following result.

\begin{cor}
There exist rational Maurer-Cartan elements of $\HGC_{1,2}'$ whose tripod coefficient is any rational number.
\end{cor}

\subsection{Actions on the sets of Maurer-Cartan elements}
Let us finally mention the following result, which is an immediate corollary of Corollary \ref{cor:torsors}.

\begin{prop}
Let $(\HGC_{n,n}^L)^+\subset \HGC_{n,n}^L$ be the dg Lie subalgebra of series of graphs with vanishing coefficient of the line graph \eqref{equ:linegraph}. Similarly, let $(\HGC_{n-1,n}^T)^+\subset \HGC_{n-1,n}^T$ be the dg Lie subalgebra and $(\HGC_{1,2}')^{T'})^+\subset \HGC_{1,2}')^{T'}$ be the $L_\infty$-subalgebras of series of graphs with vanishing tripod coefficient.
Then the exponential group $\Exp(H^0(\GC_n))$ acts freely transitively on the spaces of gauge equivalences of Maurer-Cartan elements $\oMC((\HGC_{n,n}^L)^+)$ and $\oMC((\HGC_{n-1,n}^T)^+)$ for $n\geq 2$. The exponential group $\Exp(H^0(\GC_2))$ acts freely transitively on $\oMC(((\HGC_{1,2}')^{T'})^+)$.
\end{prop}

\end{document}